\documentclass[twoside,leqno,10pt]{amsart}

\usepackage{amsfonts}
\usepackage{amsmath}
\usepackage{amscd}
\usepackage{amssymb}
\usepackage{amsthm}
\usepackage{latexsym}
\usepackage{xcolor}
\usepackage{color}

\newtheorem{theorem}{Theorem}
\newtheorem{lemma}{Lemma}
\newtheorem{proposition}{Proposition}
\newtheorem{corollary}{Corollary}

\theoremstyle{definition}

\newtheorem{remark}{Remark}

\newcommand{\intz}{\mathbb{Z}}

\newcommand{\natn}{\mathbb{N}}

\newcommand{\rear}{\mathbb{R}}

\title{An uncountable subring of $\mathbb{R}$ with Hausdorff dimension zero}

\author{Stephan Baier}
\address{Stephan Baier,
Department of Mathematics, Ramakrishna Mission Vivekananda Educational and Research Institute, PO Belur Math, Howrah, West Bengal 711202, India}
\email{stephanbaier2017@gmail.com}

\author{Shameek Paul}
\address{Shameek Paul ,
Department of Mathematics, Ramakrishna Mission Vivekananda Educational and Research Institute, PO Belur Math, Howrah, West Bengal 711202, India}
\email{shameek.rkmvu@gmail.com}

\begin{document}

\begin{abstract}
We construct a subring as mentioned in the title (hence this subring has Lebesgue measure zero).
\end{abstract}

\maketitle

\section{Preliminaries}

For a subset $B\subseteq \mathbb R$ we let $d(B)$ denote the diameter of $B$. Given a countable collection of subsets $\mathcal C\subseteq \mathcal P(\mathbb R)$ and $s\geq 0$ we define
\[d(s,\mathcal C):=\sum_{B\in \mathcal C} d(B)^s.\]
We note that $d(s,\mathcal C)\in [0,\infty]$ and does not depend on the ordering of the subsets in the collection $\mathcal C$. For $\delta>0$, we say that a collection of subsets $\mathcal C$ is a {\it $\delta$-cover} of a subset $F\subseteq \mathbb R$ if for every $B\in \mathcal C$ we have $d(B)\leq \delta$ and if $F\subseteq \cup\, \mathcal C$.

Given a subset $F\subseteq\mathbb R$ and $s,\delta\in\mathbb R$ with $s\geq 0$, and $\delta >0$, we define
\[\mathcal H^s_\delta(F)=\inf\big\{d(s,\mathcal C):\mathcal C\mbox{ is a countable $\delta$-cover of $F$}\big\}.\]
If $\delta_1<\delta_2$, then we see that $\mathcal H^s_{\delta_1}(F)\geq \mathcal H^s_{\delta_2}(F)$. The {\it $s$-dimensional Hausdorff measure $\mathcal H^s(F)$} is defined to be the limit of $\mathcal H_\delta^s(F)$ as $\delta\to 0^+$.

It can be shown (see \cite{F2}) that there exists $\alpha\in [0,1]$ such that for every $s\in [0,\alpha)$ we have $\mathcal H^s(F)=\infty$, and for every $s >\alpha$ we have $\mathcal H^s(F)=0$. The number $\alpha$ is called the {\it Hausdorff dimension} of $F$ and is denoted by $\dim_\mathcal H F$. If $F=\cup\,\mathcal C$ where $\mathcal C$ is a countable collection of subsets of $\mathbb R$, then $\dim_\mathcal H F=\sup\big\{\dim_\mathcal H A:A\in\mathcal C\big\}$.

\section{Construction}

We give an explicit construction of a proper subring of $\mathbb R$ which is a Borel set. We also show explicitly that it has Hausdorff dimension zero.

We let $\mathbb N_0$ denote the set of non-negative integers, $\mathbb N$ denote the set of positive integers, and $\mathbb Z$ denote the set of integers. For every $a,b\in \mathbb Z$ we let $[a,b]_\mathbb Z=\{k\in \mathbb Z:a\leq k\leq b\}$. Let
$$
S:=\{0,1,2,2^2,2^3,\ldots\}=\{0\}\cup \{2^n : n\in \natn_0\}.
$$
For every $n\in\mathbb N$ we let the {\it $n^\text{th}$ sumset of $S$} be defined as
\[
nS:=\{s_1+\cdots +s_n: s_1,\ldots,s_n\in S\}.
\]
For every $n\in \mathbb{N}$ we define
\[
A_n:=\left\{\sum\limits_{k\in nS} \frac{x_k}{2^k}:\mbox{ there exists }t\in \mathbb N \mbox{ such that }x_k\in [-t,t\,]_\mathbb Z\mbox{ for every }k\in nS\right\}.
\]

\begin{theorem}\label{group}
Let $n\in\natn$. The set $A_n$ is an uncountable subgroup of $\rear$.
\end{theorem}

\begin{proof}
We first show that $A_n$ is a subgroup of $\rear$. Trivially, we have $0\in A_n$, and that $-x\in A_n$ whenever $x\in A_n$. Thus, it suffices to show that $A_n$ is closed under addition.  Let $x,y\in A_n$. Therefore, there exists $t\in \mathbb N$ and for every $k\in nS$ there exist $x_k,y_k\in [-t,t\,]_\mathbb Z$ such that
\[
x=\sum\limits_{k\in nS} \frac{x_k}{2^k} \quad \mbox{and} \quad y=\sum\limits_{k\in nS} \frac{y_k}{2^k}.
\]
It follows that
$$
x+y=\sum\limits_{k\in nS} \frac{x_k+y_k}{2^k}
$$
is also an element of $A_n$ since $x_k+y_k\in [-2t,2t\,]_\mathbb Z$ for every $k\in nS$. Also, the set $A_n$ is uncountable, because by the uniqueness of binary expansions which do not end with an infinite string of ones, the set of real numbers of the form
$$
\sum\limits_{k\in nS} \frac{z_k}{2^k} \quad \mbox{with } z_k\in \{0,1\}\mbox{ for every }k\in nS
$$
is an uncountable subset of $A_n$.
\end{proof}

We now proceed to show that the set $A_n$ has Hausdorff dimension zero for every $n\in\natn$. For every $n\in\natn$ write
\[
A_n:=\bigcup_{t=1}^{\infty} A_{n,t}
\]
where for every $t\in\natn$ we let
$$
A_{n,t}:=\left\{\sum\limits_{k\in nS} \frac{x_k}{2^k}\in \mathbb{R} :\mbox{ for every } k\in nS\mbox{ we have }x_k\in [-t,t\,]_\mathbb{Z}\right\}.
$$

For every $n,\ell\in\mathbb N$ we let $(nS)_{\leq \ell}=\{k\in nS:k\leq \ell\}$ and we let $(nS)_{>\ell}=\{k\in nS:k>\ell\}$. We will need the following result to show that the set $A_{n,t}$ has Hausdorff dimension zero for every $n,t\in\natn$.

\begin{lemma} \label{density}
For every $n,\ell\in \mathbb N$ we define
$$
g_n(\ell):=\emph{\texttt{\#}}(nS)_{\le \ell}.
$$
Then we have the bound
$$
g_n(\ell) \le (2+\log_2 \ell)^n.
$$
\end{lemma}

\begin{proof}
We observe that for every $z\in\mathbb N$, we have
{
\[
g_n(\ell)\le g_1(\ell)^n\mbox{ and }g_1(\ell)\leq 2+\log_2 \ell.
\]
}
The result follows.
\end{proof}

\begin{theorem}\label{dimension}
The set $A_{n,t}$ has Hausdorff dimension zero for all $n,t\in \mathbb{N}$.
\end{theorem}

\begin{proof}
Let $x\in A_{n,t}$. Then for every $k\in nS$ there exist $a_k\in [-t,t\,]_\mathbb Z$ such that
$$
x=\sum\limits_{k\in nS} \frac{a_k}{2^k}.
$$
Pick $\ell\in \mathbb N$ and write $x$ as
$$
x=y+z,
$$
where
$$
y:=\sum\limits_{k\,\in\, (nS)_{\leq \ell}}  \frac{a_k}{2^k} \quad \mbox{and} \quad
z:=\sum\limits_{k\,\in\, (nS)_{> \ell}}  \frac{a_k}{2^k}.
$$
Then using the triangle inequality, we have the bound
$$
|z|\le t\sum\limits_{k\,\in\, (nS)_{> \ell}}  \frac{1}{2^k}\le t\sum\limits_{k=\ell+1}^{\infty} \frac{1}{2^k}=\frac{t}{2^\ell}.
$$
Let
$$
B:=\left\{\sum\limits_{k\,\in\, (nS)_{\leq \ell}}  \frac{a_k}{2^k}: \mbox{ for every }k\in (nS)_{\leq \ell}\mbox{ we have }a_k\in [-t,t\,]_\mathbb Z\right\}
$$
and let $r=t/2^\ell$. It follows that
$$
A_{n,t}\subseteq \bigcup_{b\in B} I_b\mbox{ where }I_b=[b-r,b+r]\mbox{ for every }b\in B.
$$
Let $c_\ell=\emph{\texttt{\#}}\,(nS)_{\leq \ell}$.
Then $\emph{\texttt{\#}}B\leq (2t+1)^{c_\ell}$ and for every $b\in B$ we see that $\mu(I_b)=2r$ where $\mu$ is the Lebesgue measure on $\mathbb R$. Let $s\in (0,1]$. We see that
$$
\mathcal{H}_{2r_\ell}^s(A_{n,t}) \le \sum\limits_{b\in B} \mu(I_b)^s=\emph{\texttt{\#}}B\,(2r)^s\leq (2t+1)^{c_\ell}\frac{t^s}{2^{(\ell-1)s}}.
$$
We claim that
\begin{equation} \label{limit}
\lim\limits_{\ell\rightarrow \infty} (2t+1)^{c_\ell}\frac{t^s}{2^{(\ell-1)s}}=0.
\end{equation}
Since $r\to 0$ as $\ell\to \infty$, this implies that $\mathcal H^s(A_{n,t})=0$. As this holds for all $s\in (0,1]$, we conclude that the Hausdorff dimension of $A_{n,t}$ is zero.

Equation \eqref{limit} above is seen as follows. First, it is convenient to note that \eqref{limit} holds if
$$
\lim\limits_{\ell\to\infty} \frac{(3t)^{c_\ell}}{2^{\ell s}}=0.
$$
Now, by taking logarithms, we see that \eqref{limit} holds if
\[
\lim\limits_{\ell\to \infty} \Big(\ell s-c_\ell\log_2(3t)\Big)=\infty.
\]
Let $\alpha=\frac{\log_2(3t)}{s}$. As $s>0$, we see that \eqref{limit} holds if
\begin{equation} \label{inf}
\lim\limits_{\ell\to \infty} \Big(\ell-\alpha c_\ell\Big)=\infty.
\end{equation}
Now using Lemma \ref{density}, we have
$$
c_\ell \le (2+\log_2 \ell)^n,
$$
which implies
$$
\ell\ge 2^{c_\ell^r-2}
$$
where $r=1/n$. Hence, it follows that \eqref{inf} holds. This justifies our claim.
\end{proof}

Let $n\in\natn$. From Theorems \ref{group} and \ref{dimension} it follows that the set $A_n$ is an uncountable subgroup of $\rear$ of Hausdorff dimension zero. As $0\in S$, we have $nS\subseteq (n+1)S$, and hence $A_n\subseteq A_{n+1}$. Let
\[
A:=\bigcup_{n=1}^{\infty} A_n.
\]
We deduce that the set $A$ is an uncountable subgroup of $\rear$ of Hausdorff dimension zero. We aim to establish that $A$ is in fact a subring. To this end, we need the following result.

\begin{lemma} \label{representations}
Let $n\in \mathbb{N}$ and let $c_n(k)$ be the number of representations of an element $k\in \natn$ as a sum $k=s_1+s_2+\cdots +s_n$ of $n$ elements of $S$. Then there exists $b_n\in \mathbb N$ such that $c_n(k)\leq b_n$ for every $k\in \natn$.
\end{lemma}

\begin{proof}
Clearly $c_1(k)\le 1$ for every $k\in\natn$ and hence, we may take $b_1:=1$. Now assume that $n\ge 2$ and that there exists $b_{n-1}\in\mathbb N$ such that $c_{n-1}(k)\le b_{n-1}$ for every $k\in \mathbb{N}$. Let $k\in \natn$ and suppose
$$
k=s_1+\cdots +s_n
$$
is a representation of $k$ as a sum of $n$ elements of $S$. Then for at least one $i\in\{1,\ldots,n\}$ we have $s_i\ge k/n$ and hence $s_i\neq 0$. So there exists $m\in \natn_0$ such that $s_i=2^m$. As $k/n\le s_i\le k$, we observe that
$\log_2 k-\log_2 n\le m\le \log_2 k$. Let
\[I_{n,k}:=[\,\log_2 k-\log_2 n,\, \log_2 k\,]\cap \natn_0.\]
Then we see that
$$
c_n(k)\le n\sum\limits_{m\in I_{n,k}} c_{n-1}(k-2^m) \le nb_{n-1}\emph{\texttt{\#}}I_{n,k}\le nb_{n-1}(1+\log_2 n).
$$
So we are done by taking $b_n\in \mathbb N$ such that $b_n\geq nb_{n-1}(1+\log_2 n)$.
\end{proof}

\begin{theorem} \label{main}
The set $A$ is a a subring of $\mathbb{R}$.
\end{theorem} 

\begin{proof}
It suffices to show that $1\in A$ and that $A$ is closed under multiplication. We see that
$$
1=\frac{1}{2^0}+\sum\limits_{k\in S_{>0}} \frac{0}{2^k}\in A_1.
$$
To establish that $A$ is closed under multiplication, we proceed as follows. Let $x,y\in A$. There exist $m,n\in \mathbb{N}$ such that $x\in A_m$ and $y\in A_n$. Therefore, there exist $t,t'\in \mathbb N$ and for every $k\in nS$ there exist $x_k\in [-t,t\,]_\mathbb Z$ and $y_k\in [-t',t']_\mathbb Z$ such that
\[
x=\sum\limits_{k\in mS} \frac{x_k}{2^k} \quad \mbox{and} \quad y=\sum\limits_{\ell\in nS} \frac{y_\ell}{2^\ell}.
\]
Now
\[
xy=\left(\sum\limits_{k\in mS} \frac{x_k}{2^k}\right) \left(\sum\limits_{\ell\in nS} \frac{y_\ell}{2^\ell}\right)=
\sum\limits_{(k,\ell)\in mS\times nS} \frac{x_ky_\ell}{2^{k+\ell}}=\sum\limits_{r\in (m+n)S} \frac{z_r}{2^r}
\]
where for every $r\in (m+n)S$ we let
\[
z_r:=\sum\limits_{\substack{(k,\ell)\in mS\times nS \smallskip \\ k+\ell=r}}  x_ky_\ell.
\]
The number of representations of an element $r\in (m+n)S$ as a sum $r=k+\ell$ of elements $k\in mS$ and $\ell\in nS$ is bounded by the number of representations of $r$ which are of the form $r=s_1+s_2+\cdots +s_{m+n}$ where $s_1,s_2,\ldots,s_{m+n}\in S$. Hence, by Lemma \ref{representations} we deduce that there exists $b\in\mathbb N$ such that for every $r\in (m+n)S$ we have $|z_r|\leq btt'$. It follows that $xy\in A_{m+n}$. This establishes that $A$ is closed under multiplication.
\end{proof}

We now proceed to show that  the only rational numbers in $A$ are the dyadic rationals. In particular, it will follow that the ring $A$ is not a field.

\begin{lemma}\label{gaps}
Let $b,n\in \mathbb N$. Then there exist two consecutive elements $k,k'\in nS$ such that $k'-k>b$.
\end{lemma}

\begin{proof}
For every $\ell\in \mathbb N$ we let $c_\ell:=\emph{\texttt{\#}}\,(nS)_{\le \ell}$. If the conclusion is false, then $c_\ell\geq \ell/b$ for every $\ell\in\mathbb N$. This contradicts Lemma \ref{density}.
\end{proof}

\begin{proposition}\label{string}
Let $x\in A$ and $\ell\in\mathbb N$. Then there exists a binary expansion of $x$ which contains either a string of zeroes or a string of ones having length $\ell$.
\end{proposition}

\begin{proof}
Let $x\in A_{n,t}$. Then for every $k\in nS$ there exists $x_k\in [-t,t\,]_\mathbb Z$ such that
$$
x=\sum\limits_{k\in nS} \frac{x_k}{2^k}.
$$
Let $m\in nS$. We observe that
\[2^mx=q_m+r_m\]
where
\[q_m:=\sum\limits_{k\in (nS)_{\leq m}} 2^{m-k}x_k\mbox{ and
 }r_m:=
\sum\limits_{k\in (nS)_{>m}} \frac{x_k}{2^{k-m}}.
\]
Then $q_m\in\intz$ and we have
$$
|r_m|\leq ~t\sum\limits_{j=m'-m}^{\infty} \frac{1}{2^j} =\frac{2t}{2^{m'-m}}
$$
where $m,m'\in nS$ are consecutive elements.

By Lemma \ref{gaps}, there exist consecutive elements $k,k'\in nS$ such that $2^{k'-k}>2^{\ell+1}t$ and so we get $|r_k|<\frac{1}{2^\ell}$.
Hence,
$$
\text{frac}\left(2^kx\right)=
\begin{cases}
r_k, & \mbox{ if } r_k\in \left[\,0,\frac{1}{2^\ell}\right);\\
1+r_k, & \mbox{ if } r_k\in \left(-\frac{1}{2^\ell},0\right);
\end{cases}
$$
and so
\[
\text{frac}(2^kx)\in \left[0,\tfrac{1}{2^\ell}\right)\cup \left(1-\tfrac{1}{2^\ell},1\right),
\]

\bigskip
\noindent
where frac$(y):=z-\lfloor y\rfloor$ is the fractional part of $y\in \mathbb{R}$.
It follows that the digits after the dot in a binary expansion of $x$ from the $(k+1)^\text{th}$ position to the $(k+\ell)^\text{th}$ position are all equal.
\end{proof}

\begin{corollary}
A number $x\in A$ is rational if and only if there exists $a\in \mathbb Z$ and $k\in \mathbb N_0$ such that $x=a/2^k$.
\end{corollary}

\begin{proof}
From Proposition \ref{string}, we see that the ring $A$ only contains those rational numbers which have a finite binary expansion. The result follows.
\end{proof}

\begin{remark} The above construction of an uncountable subring of $\mathbb{R}$ of Hausdorff dimension zero goes through if the initial set $S=\{0,1,2,2^2,2^3,\ldots\}$ is replaced by a set $T\cup\{0\}$ where $T$ is a subset of $\natn$ which has the following property:
There exists $b\in \rear$ such that for every $n\in\natn$, the cardinality of the set $\{\log t:t\in T\}\cap [n,n+1)$ is less than $b$.
\end{remark}

\section{Concluding remarks}

In 1966, Erd\H{o}s and Volkmann \cite{EV} showed that for every $d\in [0,1]$ there exists a subgroup of $\mathbb R$ which is Borel measurable and has Hausdorff dimension $d$. (Example 12.4 of \cite{F2}). In an unpublished work (see \cite{F1}, pg. 212), using the Continuum Hypothesis, R. O. Davies showed that for every $d\in [0,1]$ there exists a subring of $\mathbb R$ which has Hausdorff dimension $d$ and which is not a Borel set. In \cite{EM}, it is shown that every proper subring of $\mathbb R$ which is a Borel set, has Hausdorff dimension zero.

It has come to our attention that our construction has some similarities to that of \cite{EV}, but our series expansion of real numbers is inspired by the binary expansion. Erd\H{o}s and Volkmann \cite{EV} noted that the question of the Hausdorff dimension of subrings and subfields remained open. We feel that it is of interest to see how Lemma \ref{representations} is crucially used to show that we get a subring of $\mathbb R$ by using numbers which are represented by special series expansions. Also, we are able to see relatively easily that the set $A$ has Hausdorff dimension zero.

We have been informed of the following construction. If $A$ is a compact Cantor set such that for every $n\in \mathbb N$ the set $A^n$ has Hausdorff dimension zero, then it can be shown that the subring of $\mathbb R$ which is generated by $A$ has Hausdorff dimension zero. However, we are unable to find a reference for the construction of such a set. Our construction is very explicit using series expansions and we use neither the Continuum Hypothesis nor the Axiom of choice. 

\bigskip



\begin{thebibliography}{10}

\bibitem{EM}
G. A. Edgar and C. Miller, Borel subrings of the reals, {\it Proc. Amer. Math. Soc.} {\bf 131(4)} (2002), 1121--1129.

\bibitem{EV}
P. Erd\H{o}s and B. Volkmann, Additive Gruppen mit vorgegebener Hausdorffscher Dimension, {\it J. Reine Angew. Math} {\bf 221} (1966), 203--208.

\bibitem{F1}
K. Falconer, On the Hausdorff dimensions of distance sets, {\it Mathematika} {\bf 32} (1985), 206--212.

\bibitem{F2}
K. Falconer, {\it Fractal geometry, Mathematical foundations and applications \emph{(}Second Edition\emph{)}}, John Wiley \& Sons Inc., Hoboken, NJ, 2003.
\end{thebibliography}
\end{document}